\documentclass[11pt, twoside]{article}
\usepackage{amsmath,amsthm,amssymb}
\usepackage{times}
\usepackage{enumerate}
\usepackage{bm}
\usepackage[all]{xy}
\usepackage{mathrsfs}
\usepackage{amscd}
\usepackage{color}
\def\titlerunning#1{\gdef\titrun{#1}}
\makeatletter
\def\author#1{\gdef\autrun{\def\and{\unskip, }#1}\gdef\@author{#1}}
\def\address#1{{\def\and{\\\hspace*{18pt}}\renewcommand{\thefootnote}{}%
\footnote {#1}}%
\markboth{\autrun}{\titrun}}
\makeatother
\def\email#1{e-mail: #1}
\def\subjclass#1{{\renewcommand{\thefootnote}{}%
\footnote{\emph{Mathematics Subject Classification (2020):} #1}}}
\def\keywords#1{\par\medskip
\noindent\textbf{Keywords.} #1}

\newtheorem{theorem}{Theorem}[section]
\newtheorem{corollary}[theorem]{Corollary}
\newtheorem{lemma}[theorem]{Lemma}
\newtheorem{proposition}[theorem]{Proposition}
\theoremstyle{definition}
\newtheorem{definition}[theorem]{Definition}
\newtheorem{remark}[theorem]{Remark}

\numberwithin{equation}{section}

\xyoption{all}

\def \R {\mathbb{R}}

\def \R {\mathbb{R}}

\begin{document}
\baselineskip=17pt

\titlerunning{The Dirichlet eigenvalue problems for some concave elliptic Hessian operators}
\title{The Dirichlet eigenvalue problems for some concave elliptic Hessian operators}

\author{Jiaogen Zhang}

\date{}

\maketitle
 
\address{J. Zhang: School of Mathematics, Hefei University of Technology,
Hefei, Anhui, 230009, P. R. China;
\email{zjgmath@hfut.edu.cn}}

%\author[T. Zheng]{Tao Zheng}
%\address{Tao Zheng; School of Mathematics and Statistics, Beijing Institute of Technology, Beijing 100081, China;
%\email{zhengtao@amss.ac.cn}}


\begin{abstract}
In this manuscript, we investigate a priori estimates for the solution to the Dirichlet eigenvalue problem for a broad class of concave elliptic Hessian operators of the form
\[
F(D^2u)=-\Lambda u \quad \textrm{in} \, \Omega, \qquad u=0 \quad \textrm{on} \, \partial \Omega.
\]
These operators encompass the Monge-Amp\`ere operator, the $k$-Hessian operators, and the $p$-Monge-Amp\`ere operators. We impose a  fairly  mild constraint on the operator $F$,  allowing us to demonstrate the existence of the first nonzero eigenvalue and its  corresponding $\Gamma$-admissible eigenfunction on the smooth, strictly $\Gamma$-convex domain $\Omega\subset \mathbb{R}^{n}$. Furthermore,  we prove that the eigenfunction $u_{1}$ belongs to $C^{\infty}(\Omega) \cap C^{1,1}(\overline{\Omega})$. As an application, we prove that every invariant G\r{a}rding-Dirichlet operator admits a unique first nonzero eigenvalue. Finally, a bifurcation-type theory for these operators is also established.
\end{abstract}


\section{Introduction}
Let $\Gamma \subset \mathbb{R}^n$ (for $n \geq 2$) be an open \textit{convex invariant cone}, and let $f : \Gamma \rightarrow \mathbb{R}$ be a \textit{concave elliptic Hessian operator} (as defined in Definitions \ref{CIC} and \ref{elliptic hessian operator}). We define the corresponding set of real symmetric $n\times n$ matrices characterizing as
\[\mathscr{M}(\Gamma)=\{A\in \mathrm{Sym}^2(\mathbb{R}^n)\, | \, \lambda(A)\in \Gamma\}\] 
where $\lambda(A)$ denotes the eigenvalue vector of $A$. For an arbitrary bounded domain $\Omega\subset \mathbb{R}^n$, a smooth function $u$ in $\Omega$ for which $D^2u(x)\in \mathscr{M}(\Gamma)$ for all $x\in \Omega$ is said to be $\Gamma$-admissible. For notational convenience, we set
\[
F(A)=f(\lambda(A))\,, \qquad A\in \mathscr{M}(\Gamma).
\]
Our main objective is to establish the existence and regularity of a $\Gamma$-admissible eigenpair $(u_1,\Lambda_{1})$ 
satisfying the following Dirichlet eigenvalue problem
\begin{equation}\label{eigenvalue problem}
\left\{\begin{array}{ll}
\begin{matrix}
\begin{aligned}
F(D^2u) =&\, -\Lambda u  \\
 u= &\, 0
\end{aligned} 
&
\begin{aligned}
\quad &\mbox{in } \Omega,  \\
\quad &\mbox{on } \partial\Omega,
\end{aligned} 
\end{matrix}
\end{array}\right.
\end{equation}
where $D^2u$ denotes the Hessian matrix of $u$.  A notable characteristic of this equation is its degeneracy near the boundary, as the right-hand side tends to zero.

Fully nonlinear Dirichlet eigenvalue problems form a deep and interdisciplinary area, bridging fully nonlinear PDE theory, spectral geometry, and the calculus of variations. Their complexity stems fundamentally from the nonlinearity of the operator $F$—a phenomenon with no analogue in the classical spectral theory of the Laplacian. Furthermore, the Dirichlet eigenvalue problem for such operators reveals deep geometric content through its connection to the Brunn-Minkowski inequality in convex geometry. For further reading, we refer to \cite{CF20, LMX, Salani05}, among other works.

The bibliography related to the Dirichlet eigenvalue problem for fully nonlinear second order operators is very wide. A canonical example is the Dirichlet eigenvalue problem for the Monge-Amp\`ere operator. This problem consists in finding a nontrivial pair  $(u_1,\Lambda_1)\in C(\overline{\Omega})\times \mathbb{R}$ such that
\begin{equation*}
\left\{\begin{array}{ll}
\begin{matrix}
\begin{aligned}
\det(D^2u) =&\, (-\Lambda u)^{n}  \\
 u= &\, 0
\end{aligned} 
&
\begin{aligned}
\quad &\mbox{in } \Omega,  \\
\quad &\mbox{on } \partial\Omega.
\end{aligned} 
\end{matrix}
\end{array}\right.
\end{equation*}
A landmark systematic study of this problem was undertaken by Lions~\cite{L85}, who established the existence of the first eigenvalue $\Lambda_1$, proved the regularity $u_1 \in C^{\infty}(\Omega) \cap C^{1,1}(\overline{\Omega})  $ for the corresponding eigenfunction, and provided a variational characterization of $\Lambda_1$.  Subsequent contributions by Tso~\cite{Tso}, Hong–Huang–Wang~\cite{HHW11}, Le–Savin~\cite{LS17}, and Le~\cite{L18,L25} established existence, uniqueness, and geometric regularity properties under appropriate convexity constraints.

Wang's seminal work \cite{W94}  extended the eigenvalue framework to the $k$-Hessian operators
\[
\sigma_{k}(D^2u)=\sum_{1 \leq i_1 < i_2 < \cdots < i_{k} \leq n} \lambda_{i_1} \lambda_{i_2}  \cdots \lambda_{i_k}\,, \qquad \lambda=\lambda(D^2u)\,,
\]
unifying Monge-Amp\`ere theory with a broader class of fully nonlinear elliptic operators. Recent progress includes the study of the $p$-Monge-Amp\`ere operators 
\begin{equation*}
\mathcal{M}_{p}(D^2u)= \prod_{1 \leq i_1 < i_2 < \cdots < i_{p} \leq n} (\lambda_{i_1} + \lambda_{i_2} + \cdots + \lambda_{i_p})\,, \qquad \lambda=\lambda(D^2u)\,,
\end{equation*}
with the critical case $p = n-1$ resolved in~\cite{Huang,M21}, and  completion of all integer cases $1 \leq p \leq n$  by the author and Zheng \cite{ZZ25}.  
 
We now describe the several fundamental structures for $f$. Let $\Gamma$ be a convex invariant cone with vertex at the origin satisfying 
\[
\Gamma_{n}=\{\lambda\in \mathbb{R}^n\, :\, \lambda_{i}>0, \, i=1,\cdots,n\}\subset \Gamma \subset \{\lambda\in \mathbb{R}^n\, :\, \sum_{i=1}^{n}\lambda_{i}>0\}=\Gamma_{1}\,.
\]
We shall assume further that $f$ satisfies the following condition on $\Gamma$.

\begin{definition}
We say an elliptic operator $f: \Gamma \rightarrow \mathbb{R}$ satisfies the Condition {\bf (T)} if there exists a uniform  constant  $\mathscr{T}>0$ such that  
\begin{equation}\label{det}
\prod_{i=1}^{n}\frac{\partial f}{\partial \lambda_{i}}\geq  {\mathscr{T}}\,.
\end{equation}  
\end{definition}
Heuristically, Condition {\bf (T)} allows us exploit an established upper bound for $\{\frac{\partial f}{\partial \lambda_{i}}\}_{i=1}^{n}$ to derive a lower bound, thereby providing two-sided control. This type of argument is common in a priori estimates when the maximum principle is applied to the linearized operator. We note that this hypothesis is also employed in the recent study by Guo-Phong \cite{GP24}, which investigates $L^{\infty}$ estimates for fully nonlinear equations on non-K\"ahler manifolds.

Our first primary result of this manuscript is the following existence theorem.
\begin{theorem}\label{eigenvalue theorem}
Let $\Omega$ be a smooth bounded and strictly $\Gamma$-convex domain within $\mathbb{R}^n$. Assume $F$ is a concave elliptic Hessian operator on an open convex cone $\mathscr{M}(\Gamma)\subset \mathrm{Sym}^2(\mathbb{R}^n)$ that satisfies the Condition {\bf (T)}.  Then, there exists a positive number $\Lambda_{1}=\Lambda_{1}(\Omega)$ such that \eqref{eigenvalue problem} possesses a $\Gamma$-admissible solution $u_{1}\in C^{\infty}(\Omega)\,\cap\, C^{1,1}(\overline{\Omega})$.
Furthermore, the solution is unique up to scalar multiplication.
\end{theorem}
\begin{remark}\label{example}
\begin{itemize}\itemsep 0.5em
\item[(1)] The positive constant $\Lambda_{1}$ will be referred to as the first eigenvalue of the operator $F$, and $u_1$ will be referred to as its corresponding eigenfunction.
\item[(2)] Examples of $f$ satisfying Condition  {\bf (T)} including 
\begin{itemize}
\item[(i)] $(\Gamma_{k},\sigma_{k}^{\frac{1}{k}})$, the $k$-Hessian operator for $1\leq k\leq n$, where $\sigma_{k}$ is the $k$-th elementary symmetric polynomial 
\[
\sigma_{k}(\lambda)=\sum_{1\leq i_1<i_2<\cdots <i_{k}\leq n}\lambda_{i_1}\cdots \lambda_{i_{k}}
\]
and $\Gamma_{k}=\{\lambda\in \mathbb{R}^{n}\,|\, \sigma_{j}(\lambda)>0, \, j=1,\cdots, k\}.$ The proof of Condition {\bf (T)} can be found in the work of Wang \cite[p. 183]{Wang}.
\item[(ii)] $(\mathcal{P}_{p},\mathcal{M}^{1/\binom{n}{p}}_{p})$, the $p$-Monge-Amp\`ere operator for $1\leq p\leq n$, where $\mathcal{M}_{p}$ is a symmetric polynomial given by
\begin{equation*}
\mathcal{M}_{p}(\lambda)=\prod_{1\leq i_1<i_2<\cdots <i_{p}\leq n} (\lambda_{i_1}+\lambda_{i_2}+\cdots +\lambda_{i_p})\,.
\end{equation*}
and $\mathcal{P}_{p}=\{ \lambda\in \R^n \,: \,\lambda_{i_1}+\cdots +\lambda_{i_{p}}>0,\, 1\leq i_1<\cdots< i_{p}\leq n\}$.  The proof of Condition {\bf (T)} has been provided in \cite{D21,ZZ25}. 
\end{itemize} 
\item[(3)] Notably for all $n\geq k>l\geq 1$, the Hessian quotient equations  $(\Gamma_{k}, (\frac{\sigma_{k}}{\sigma_{l}})^{1/(k-l)})$  fail to meet the  Condition {\bf (T)}.
\end{itemize}
\end{remark}

In a recent work, Harvey and Lawson \cite{HL23} showed that the class of operators satisfying Condition \textbf{(T)} is unexpectedly broad, and actually includes all invariant G\r{a}rding-Dirichlet operators. Moreover, they established the following strict ellipticity property on the G\r{a}rding cone $\Gamma$:
\begin{equation}\label{GD}
f(\lambda+\tau)\geq f(\lambda)+\Big(\prod_{i=1}^{n}  \tau_i\Big)^{\frac{1}{n}} \qquad \text{for all } \lambda\in \Gamma\, \text{ and } \tau\in \Gamma_{n}\,.
\end{equation}
In Lemma \ref{GD 2} below, we prove that inequality \eqref{GD} implies \eqref{det}. Combining this with Theorem \ref{eigenvalue theorem} yields the following corollary.
\begin{corollary}
Let $\Omega$ be a smooth bounded and strictly $\Gamma$-convex domain within $\mathbb{R}^n$. Assume $F$ is an invariant G\r{a}rding-Dirichlet operator on an open convex cone $\mathscr{M}(\Gamma)\subset \mathrm{Sym}^2(\mathbb{R}^n)$.  Then, there exists a positive number $\Lambda_{1}=\Lambda_{1}(\Omega)$ such that \eqref{eigenvalue problem} possesses a $\Gamma$-admissible solution $u_{1}\in C^{\infty}(\Omega)\,\cap\, C^{1,1}(\overline{\Omega})$.
Furthermore, the solution is unique up to scalar multiplication.
\end{corollary}

To prove Theorem \ref{eigenvalue theorem}, it is necessary to investigate the following more general Dirichlet problem:
\begin{equation}\label{Dirichlet problem}
F(D^2u)=\psi(x,u)  \quad \textrm{in} \, \Omega, \qquad u=0 \quad \textrm{on} \, \partial \Omega.
\end{equation}
This problem represents a special case of the form:
\begin{equation}\label{Dirichlet problem II}
F(D^2u) = \psi(x,u,\nabla u)\,.
\end{equation}
The Dirichlet problem for equations of the type \eqref{Dirichlet problem II} has received considerable attention over the past four decades. Under various assumptions, contributions have been made by  Ivochkina \cite{Ivo}, Cafferalli-Nirenberg-Spruck \cite{CNS85}, and subsequent works in \cite{CD,CW,D21,Guan94,JW24,Li90,TW99,W94}, among many others. Trudinger \cite{Tru87} investigated the Dirichlet and Neumann problems in balls for the degenerate case, while Urbas \cite{Urbas} studied nonlinear oblique derivative boundary value problems in two dimensions.

The second main result of this manuscript addresses the solvability of the Dirichlet problem \eqref{Dirichlet problem} in the degenerate case.
\begin{theorem}\label{existence theorem}
Let $\Omega\subset \mathbb{R}^n$ be a bounded  domain with $\partial\Omega \in C^{3,1}$.  Assume $F$ is a concave elliptic Hessian operator on $\mathscr{M}(\Gamma)$ that satisfies the Condition {\bf (T)}. Suppose that $\psi\in C^{1,1}(\overline{\Omega}\times (-\infty,0],\mathbb{R})$ is non-negative and satisfies
\begin{equation*}
\psi(x,u)>0  \qquad \textrm{ on } \{x\in \Omega:u(x)<0\}\,.
\end{equation*}
If the Dirichlet problem \eqref{Dirichlet problem} possesses a subsolution $\underline{u}$ with $\underline{u}|_{\partial\Omega}\leq 0$ and a supersolution $\overline{u}$ with $\overline{u}|_{\partial\Omega}= 0$, then it admits a solution $u\in C^{\infty}(\Omega)\cap C^{1,1}(\overline{\Omega})$.
\end{theorem}
We recall that a  $\Gamma$-admissible function $v$ of class $C^{1,1}$ is referred to be a subsolution (resp. supersolution) of \eqref{Dirichlet problem} if
\begin{equation*}
\left\{\begin{array}{ll}
\begin{matrix}
\begin{aligned}
F(D^2v) \geq &\quad (\textrm{resp.} \leq )\quad \psi(x,v)  \\
 v\leq &\quad (\textrm{resp.} \geq )\quad 0
\end{aligned} 
&
\begin{aligned}
\quad &\mbox{in } \Omega,  \\
\quad &\mbox{on } \partial\Omega.
\end{aligned} 
\end{matrix}
\end{array}\right.
\end{equation*}
\medskip

The present manuscript is organised as follows: In Section 2, we revisit the concept of concave elliptic Hessian operators and compile a list of well-known properties concerning the derivatives of the operator $F$. 
In Sections 3 and 4, we establish both the global gradient estimate and the global second-order estimate for the problem defined in \eqref{Dirichlet problem}. These estimates enable us to demonstrate the existence result outlined in Theorem \ref{existence theorem}. 
In Section 5 we  conclude the proof of eigenvalue problem \eqref{eigenvalue problem} by employing the  Lions' approximation strategy. Finally, Section \ref{Application to the bifurcation theorem} is devoted to establishing a connection between Theorem \ref{existence theorem} and the bifurcation-type theory for a class of concave elliptic Hessian operators that satisfying Condition {\bf (T)}.
\medskip

\noindent {\bf Acknowledgements.}This work was supported by the National Natural Science Foundation of China  (Grant No. 12401065) and Initial Scientific Research Fund of Young Teachers in Hefei university of Technology  (Grant No. JZ2024HGQA0119). The author is deeply grateful to his thesis advisor Prof. Xi Zhang for his   constantly guidance and encouragement. He also wishes to thank Dr. Xinqun Mei for kindly bringing the reference \cite{Urbas91} to his attention, and Prof. Jianchun Chu for his valuable comments and suggestions.

\section{Preliminaries}\label{Preliminaries}
\subsection{Convex invariant cones and concave elliptic Hessian operators}
\begin{definition}\label{CIC}
An open cone $\Gamma\subset \mathbb{R}^{n}$  is said to be invariant if the following conditions are satisfied:
\begin{itemize}
\item[(i)] $\Gamma+\Gamma_{n}\subset \Gamma$.
\item[(ii)] $\Gamma$ remains unchanged under permutations of its coordinates.
\end{itemize}
Additionally, if $\Gamma$ is convex, it is referred to as a convex invariant cone.
\end{definition}
The following is the concept of a (strictly) $\Gamma$-convex cone that is associated with such cones.
\begin{definition}\label{p convex}
Let $\Omega$ be an open bounded domain in $\mathbb{R}^n$ with smooth boundary $\partial\Omega$. We say $\Omega$ is strictly $\Gamma$-convex if there exists a large number $R$ such that at every point $x\in \partial\Omega$
\begin{equation*}
(\kappa_1,\cdots,\kappa_{n-1},R)\in \Gamma\,,
\end{equation*}
where $(\kappa_1,\cdots,\kappa_{n-1})$ are the principle curvatures of $\partial\Omega$ at $x$.
\end{definition}
We now introduce the terminology of concave elliptic Hessian operators.
\begin{definition}\label{elliptic hessian operator}
Let $\Gamma$ be a convex invariant cone. The map  $F: \mathscr{M}(\Gamma) \longrightarrow \mathbb{R}$
is referred to as a {\em concave elliptic Hessian operator} if the following conditions are satisfied:
\begin{itemize}\itemsep 0.5em
\item[(i)] $F$ is Hessian, i.e. there exists a $C^2$ smooth function $f:\Gamma\rightarrow \mathbb{R}$ such that
\[
F(A)=f(\lambda_1(A),\cdots,\lambda_{n}(A)),
\]
where $\lambda_{i}(A)$ denotes the $i$-th eigenvalue of $A$.
\item[(ii)] $F$ is strictly elliptic, i.e. $f_{i}(\lambda):=\frac{\partial f}{\partial \lambda_{i}}>0$ for every $i$.
\item[(iii)] $F$ is concave, i.e. the matrix $\big(\frac{\partial^2 f}{\partial\lambda_{i}\partial\lambda_{j}}\big)$ is positive semi-definite.
\item[(iv)]  $f(\lambda)>0$ for all $\lambda\in \Gamma$ and $f(1,\cdots,1)=1$.
\item[(v)] $F$ is homogeneous of degree one, i.e. $f(t\lambda)=tf(\lambda)$ for every $t>0$ and for all $\lambda\in \Gamma$.
\end{itemize}
\end{definition}

It is noteworthy that the properties (i)–(v) are common to all examples presented in Remark \ref{example}.  The following lemma establishes that Condition \textbf{(T)} is satisfied by every invariant G\r{a}rding-Dirichlet operator. 
\begin{lemma}\label{GD 2}
Let $f:\Gamma\rightarrow \mathbb{R}$ be a concave elliptic operator. Assume that there exists a constant $c>0$ such that
\begin{equation}\label{inequ A}
f(\lambda+\tau)-f(\lambda)\geq c(\tau_1\cdots \tau_{n})^{\frac{1}{n}}, \qquad \text{for all }  (\lambda,\tau)\in \Gamma\times \Gamma_{n}.
\end{equation}
Then there exists a constant $c'>0$ such that
\begin{equation}\label{inequ B}
\prod_{i=1}^{n}f_{i}(\lambda)\geq c', \qquad \text{for all }  \lambda\in \Gamma.
\end{equation}
Moreover,  the converse also holds.
\end{lemma}

\begin{proof}
We first show that \eqref{inequ A} implies \eqref{inequ B}. By the concavity of $f$ together with inequality \eqref{inequ A}, we have
\[
\sum_{i=1}^{n}f_{i}(\lambda)\tau_{i}\geq \sum_{i=1}^{n}\int_{0}^{1} f_{i}(\lambda+s\tau)\tau_{i}ds=f(\lambda+\tau)-f(\lambda)\geq  c(\tau_1\cdots \tau_{n})^{\frac{1}{n}}\,.
\]
Choosing $\tau_{i} = \frac{1}{f_{i}(\lambda)}$ for each $i=1,2,\cdots,n$, we obtain
\[
n\geq c\Big(\prod_{i=1}^{n}f_{i}(\lambda)\Big)^{-\frac{1}{n}}\,.
\]
Rearranging this inequality yields \eqref{inequ B} with $c' = (c/n)^n$. 

Next, we prove that \eqref{inequ B} implies \eqref{inequ A}.  Using the integral mean value formula, there exists some $s_{0}\in [0,1]$ such that
\[
f(\lambda+\tau)-f(\lambda)=\sum_{i=1}^{n}\int_{0}^{1} f_{i}(\lambda+s\tau)\tau_{i}ds=\sum_{i=1}^{n} f_{i}(\lambda+s_{0}\tau)\tau_{i}\,.
\]
Since each term in the sum is nonnegative (as ensured by the properties of ellipticity and $\tau\in \Gamma_{n}$), it follows that
\[
f(\lambda+\tau)-f(\lambda)\geq f_{i}(\lambda+s_{0}\tau)\tau_{i} \qquad \textrm{ for all } i=1,\cdots,n.
\]
Multiplying these inequalities over $i$, we obtain
\[
\Big(f(\lambda+\tau)-f(\lambda) \Big)^{n}\geq (\tau_1\cdots \tau_{n}) \cdot \prod_{i=1}^{n} f_{i}(\lambda+s_{0}\tau)\,.
\]
Since  $\lambda+s_{0}\tau \in \Gamma$, we may apply \eqref{inequ B} to the product on the right, giving
\[
\Big(f(\lambda+\tau)-f(\lambda) \Big)^{n} \geq c' (\tau_1\cdots \tau_{n})\,.
\]
Taking the $n$-th root on both sides shows that \eqref{inequ A} holds with $c=(c')^{1/n}$.
\end{proof}

The following inequality, established by Urbas, plays a crucial role in our proof of Theorem \ref{eigenvalue theorem}.
\begin{lemma}\cite[Lemma 3.3]{Urbas91}
Let $\lambda=(\lambda_1,\cdots,\lambda_n)\in \Gamma$. Then we have
\begin{equation}\label{urbas}
f(\lambda_1,\cdots,\lambda_n)\leq \frac{1}{n}\sum_{i=1}^{n}\lambda_{i}\,.
\end{equation}
\end{lemma}

\subsection{First and second derivatives of $F$}
To establish the required a priori estimates in later sections, we proceed by differentiating Equation \eqref{Dirichlet problem} once and twice. For complete proofs of these computations, we refer the reader to \cite{Andrews94, Gerhardt96, Spr}.
\begin{lemma}
Let $u\in C^4(\Omega)$ be a solution to $F(D^2u)=\psi$. Then 
\begin{align}
\label{sum}
\sum_{i,j=1}^{n}F^{ij}(D^2u)u_{ij}&=\psi\,,\\[1mm]
\label{first derivatives}
 \sum_{i,j=1}^{n}F^{ij}(D^2u)u_{ijt}&=D_{t}\psi\,,\\[1mm]
\label{second derivatives}
 \sum_{i,j=1}^{n}F^{ij}(D^2u)u_{ijtt}+\sum_{i,j,k,l=1}^{n}F^{ij,kl}(D^2u)u_{ijt}u_{klt}&=D_{tt}\psi
\end{align}
for every $t=1,\cdots,n$, where
\[
F^{ij}(D^2u)=\frac{\partial F(D^2u)}{\partial u_{ij}}\,, \qquad F^{ij,kl}(D^2u)=\frac{\partial^2F(D^2u)}{\partial u_{ij}\partial u_{kl}}\,.
\]
If $D^2u$ is diagonal, i.e. $u_{ij}=\lambda_{i}\delta_{ij}$, we then have 
\begin{equation}\label{structures of F}
F^{ij}=\delta_{ij}f_{i}(\lambda), \quad F^{ij,kl}=\delta_{ij}\delta_{kl}f_{ik}(\lambda)+\frac{f_{i}(\lambda)-f_{j}(\lambda)}{\lambda_{i}-\lambda_{j}}(1-\delta_{ij})\delta_{il}\delta_{jk}.
\end{equation}
\end{lemma}

\section{First order estimates}\label{First order estimate}

In this section, we shall concentrate on the following global gradient estimate on the Dirichlet problem \eqref{Dirichlet problem}.
\begin{proposition}\label{gradient estimate}
Suppose $\Omega$ is a bounded and strictly $\Gamma$-convex domain in $\mathbb{R}^n$ with smooth boundary. Let $u\in C^{3}(\Omega)\cap C^{1}(\overline{\Omega})$ be a $\Gamma$-admissible solution to the Dirichlet problem \eqref{Dirichlet problem}. Then we have
\begin{equation*}
\sup_{\overline{\Omega}}|Du|\leq C_{1}\,,
\end{equation*}
where $C_{1}$ is a constant depends only on $p,$ $\Omega,$  $\|D_{x}\psi\|_{C^0}$, $\|D_{u}\psi\|_{C^0}$ and $C_{0}:= \|u\|_{C^0}$.
\end{proposition}

\begin{proof}
To commence, we need to consider the estimate of the boundary gradient. For this purpose, we shall establish bounds for the solution $u$ from both below and above.

We initially observe that $\Gamma$-admissibility ensures subharmonicity. Namely, $\sigma_{1}(D^2u)\geq 0.$ Then, by the standard comparison principle, we see that $u\leq 0$ in $\Omega$. To obtain a lower bound for $u$, we need to further construct a smooth strictly $\Gamma$-admissible subsolution $\underline{u}$ to following Dirichlet problem 
\begin{equation*}
\left\{\begin{array}{ll}
\begin{matrix}
\begin{aligned}
F(D^2v)=&\, \sup_{\overline{\Omega}\times [-C_{0},0]}\psi(x,s)  \\
 v= &\, 0
\end{aligned} 
&
\begin{aligned}
\quad &\mbox{in } \Omega,  \\
\quad &\mbox{on } \partial\Omega.
\end{aligned} 
\end{matrix}
\end{array}\right.
\end{equation*}
At present, this construction is  standard, for instance, see \cite{CNS85}. For an arbitrary point on $\partial\Omega$, after applying a translation, we may assume it is actually the origin. Close to the origin, every point $x\in \partial\Omega$ can be represented  by 
\begin{equation}\label{coordinate transfor}
x_n=\rho(x')=\frac{1}{2}\sum_{\alpha=1}^{n-1}\kappa_{\alpha}x_{\alpha}^{2}+O(|x'|^3)
\end{equation}
where $x'=(x_1,\cdots,x_{n-1})$ and $\kappa_{\alpha}$ are the principle curvature of the boundary of $\Omega$ at $0$. Around the $\partial\Omega$   we denote by
\begin{equation*}
d(x)=\min_{y\in \partial\Omega}|x-y|\,, \qquad x\in \mathbb{R}^n
\end{equation*}
be the distance function from the boundary. It is well-known that at the origin, the Hessian matrix of $d$ is given by 
\begin{equation*}
\big(d_{ij}\big)=
\begin{pmatrix}
-\kappa_1 &  & 
&  \\
 &  & \ddots
&  \\
& & 
& -\kappa_{n-1}\\
& & 
& &0
\end{pmatrix}\,.
\end{equation*}
For a large constant $R$, we define a smooth function
\begin{equation*}
v=\frac{1}{R}(e^{-Rd}-1)\,.
\end{equation*}
An easy calculation yields that at the origin
\begin{equation*}
\big(v_{ij}\big):=
\begin{pmatrix}
\kappa_1 &  & 
&  \\
 &  & \ddots
&  \\
& & 
& \kappa_{n-1}\\
& & 
& &R
\end{pmatrix}.
\end{equation*}
Recalling that $\Omega$ is strictly $\Gamma$-convex, in view of Definition \ref{p convex} and continuity, we infer that $\lambda(D^2v)\in \Gamma$ in some exterior tubular neighborhood
\[\Omega_{\delta}:=\{x\in \overline{\Omega}: 0< d(x)< \delta\}\]
for a large number $R$ and a small number $\delta>0$. Then, the construction of   subsolution  $\underline{u}$, which is strictly $\Gamma$-admissible, can be reproduced almost verbatim from \cite[Section 2]{CNS85}, we omit the standard step here. Thus, one may, once again, take  $R$ to be so large that
\begin{equation}\label{subsolution}
\left\{\begin{array}{ll}
\begin{matrix}
\begin{aligned}
F(D^2\underline{u}) \geq &\, \sup_{\overline{\Omega}\times [-C_{0},0]}\psi\geq F(D^2u)  \\
\underline{u}= &\, u=0
\end{aligned} 
&
\begin{aligned}
\quad &\mbox{in } \Omega,  \\
\quad &\mbox{on } \partial\Omega.
\end{aligned} 
\end{matrix}
\end{array}\right.
\end{equation}
\footnote{We warn the reader that the definition of $\underline{u}$ is inherently dependent on $C_{0}$, thus we cannot obtain the $C^{0}$ estimate directly.}By applying the standard comparison principle again, we infer that $\underline{u}\leq u\leq 0$ in $\Omega$. Since $\underline{u}=u=0$ on $\partial\Omega$, the boundary gradient estimate then follows immediately.  
\medskip

Now, it is time to produce the global gradient estimate. Let us define
\begin{equation*}
Q=\log\beta +\varphi(u)+\frac{a}{2}|x|^2, \qquad x\in \overline{\Omega}
\end{equation*}
where $a>0$ is a small constant to be picked soon and 
\begin{align*}
\begin{split}
\beta&=1+\frac{1}{2}|\nabla u|^2\,;\\
\varphi(t)&=-\frac{1}{3}\log (1+C_{0}-t)\,, \quad t\in [-C_0,0].
\end{split}
\end{align*}
It will be useful to observe that
\begin{equation*}
\varphi'' =3(\varphi' )^{2}\,,\qquad \varphi' =\frac{1}{3(1+C_{0}-t)}\geq \frac{1}{3(1+2C_0)} \,.
\end{equation*}

There is no loss of generality in assuming that $Q$ attains its maximum at $x_{0}\in\overline{\Omega}$. 
If $x_0$ lies on the boundary $ \partial\Omega$, then there is no need for further proof as the desired inequality is already satisfied. We now suppose that $x_0$ lies in the interior of $ \Omega$, by performing the rotation of the coordinates $x_1,\cdots,x_n$ if necessary so that the matrix $(u_{ij})$ is diagonal at $x_0$. 

To simplify the notation, let us denote
\begin{equation*}
\mathcal{F}:=\sum_{i=1}^{n}F^{ii}\,.
\end{equation*} Thus at $x_0$, we have
\begin{equation}\label{vanishes1}
0=Q_{i}=\frac{1}{\beta}u_{ii}u_{i} +\varphi'u_{i}+ax_{i}
\end{equation}
for each $i=1,\cdots,n$ as well as
\begin{equation*}
\begin{split}
0\geq \sum_{i=1}^{n}F^{ii}Q_{ii}=\frac{1}{\beta}\sum_{i,k=1}^{n}F^{ii}u_{kii}u_{k}&+\frac{1}{\beta}\sum_{i=1}^{n}F^{ii}u^{2}_{ii}-\frac{1}{\beta^{2}}\sum_{i=1}^{n}F^{ii}u_{i}^{2}u^{2}_{ii}\\
&+\varphi'\sum_{i=1}^{n}F^{ii}u_{ii}+\varphi''\sum_{i=1}^{n}F^{ii}u^{2}_{i}+a\mathcal{F}\,.
\end{split}
\end{equation*}
Owing to the monotone increasing property of $\varphi$, then by utilizing \eqref{sum}--\eqref{first derivatives} along with the inequality \eqref{vanishes1}, we can infer that
\begin{equation}\label{max principle C1}
\begin{split}
0&\geq\frac{\langle D_{x}\psi,Du\rangle+|Du|^2 D_{u}\psi}{\beta}+\frac{1}{\beta}\sum_{i=1}^{n}F^{ii}u^{2}_{ii}-\frac{1}{\beta^{2}}\sum_{i=1}^{n}F^{ii}u_{i}^{2}u^{2}_{ii}+\varphi''\sum_{i=1}^{n}F^{ii}u^{2}_{i}+\varphi'\psi(x,u)+a\mathcal{F}\\
& \geq -\frac{|D_{x}\psi|}{\sqrt{\beta}}-|D_{u}\psi|-\sum_{i=1}^{n}F^{ii}(\varphi'u_{i}+ax_{i})^{2}+\varphi''\sum_{i=1}^{n}F^{ii}u^{2}_{i}+a\mathcal{F}\\
& \geq -\frac{|D_{x}\psi|}{\sqrt{\beta}}-|D_{u}\psi|+\big(\varphi''-2(\varphi')^{2}\big)\sum_{i=1}^{n}F^{ii}u_{i}^{2}-2a^{2}\sum_{i=1}^{n}F^{ii}x_{i}^{2}+a\mathcal{F}\\
&\geq -\frac{|D_{x}\psi|}{\sqrt{\beta}}-|D_{u}\psi|+\frac{1}{9(1+2C_{0})^{2}}\sum_{i=1}^{n}F^{ii}u_{i}^{2}+\big(a-2a^{2}(diam\, \Omega)^{2} \big)\mathcal{F}\,,
\end{split}
\end{equation}
where we have used the inequality $(x+y)^2\leq 2x^2+2y^2$ in the third line. Select  a constant $a>0$ small enough such that $2a(diam\, \Omega)^{2}\leq \frac{1}{2}$. Thus, we conclude from \eqref{max principle C1} that at $x_0$
\begin{equation}\label{insert}
\frac{1}{9(1+2C_{0})^{2}}\sum_{i=1}^{n}F^{ii}u_{i}^{2}+\frac{a}{2}\sum_{i=1}^{n}F^{ii}\leq \frac{|D_{x}\psi|}{\sqrt{\beta}}+|D_{u}\psi|\leq C'
\end{equation}
for some uniform constant $C'$ provided that $\beta(x_0)\gg 1$.
Since each $F^{ii}$ is bounded above by $\frac{2C'}{a}$, and Condition {\bf (T)}  implies the lower bound
\[
F^{ii}\geq \mathscr{T}\big(\frac{a}{2C'}\big)^{n-1} \qquad  \text{for all } i=1,\cdots, n.
\]
Substituting the lower bound into \eqref{insert} yields
\[
\frac{\mathscr{T}}{9(1+2C_{0})^{2}}\big(\frac{a}{2C'}\big)^{n-1}|D u|^2\leq C'.
\]
Consequently, we deduce the required estimate for $|Du|(x_{0})$, thus concluding the proof.
\end{proof}

\section{Second order estimates}\label{Second order estimates}
The objective of this section is to establish the second order estimates. To do so, we first reduce the global second order estimate to a corresponding boundary estimate. It then suffices to prove the second order estimates at each boundary point.

\subsection{Global second order estimate}
In this subsection, we will prove the following global second order estimate:
\begin{proposition}\label{int hessian estimate}
Suppose $\Omega$ is a bounded and strictly $\Gamma$-convex domain in $\mathbb{R}^n$ with smooth boundary. Let $u\in C^{4}(\Omega)\cap C^{1,1}(\overline{\Omega})$ be a $\Gamma$-admissible solution to the Dirichlet problem \eqref{Dirichlet problem}. Then there exists a constant $C_2$ depends on $\Omega,$ $C_{0}$, $\sup_{\overline{\Omega}\times [-C_{0},0]}\psi$, $\|D_{x}\psi\|_{C^{1}}$  and  $\|D_{u}\psi\|_{C^1}$ such that
\begin{equation}\label{global hessian to boundary}
\sup_{\overline{\Omega}}|D^2u|\leq C_{2}\Big(1+\sup_{\partial\Omega}|D^2u|\Big)\,.
\end{equation}
\end{proposition}

\begin{proof}
Let $U$ denote the matrix $D^2u$ for simplicity. We consider the auxiliary function
\begin{equation*}
Q=\log \lambda_{1}(U)+\varphi\big(\frac{1}{2}|Du|^{2}\big)+\frac{1}{2}|x|^{2} 
\end{equation*}
on an open subset $\Omega_{+}:=\{x\in \Omega: \lambda_{1}(U)(x)>0\}$, where $\lambda_{1}(U)(x)$ is the largest eigenvalue of $U$ at $x$, and the function $\varphi(t)$ is given by
\begin{equation*}
\varphi(t)=-\frac{1}{2}\log(2K-t), \qquad K=1+\sup_{\overline{\Omega}}|Du|^2.
\end{equation*} 
We may assume that $\Omega_{+}$ is nonempty; otherwise, we have already done. A direct calculation yields the following:
\begin{equation}\label{phi}
\varphi''=2(\varphi')^2, \qquad \frac{1}{4K}\leq \varphi'\leq \frac{1}{2K}.
\end{equation}
Assume that $Q$ achieves its maximum at a point $x_{0}\in \Omega$. We remark that $Q$ may not be twice differentiable at $x_0$ if the eigenspace of $\lambda_{1}(U)$ has a dimension greater than $1$. To overcome this difficulty, we will apply a standard perturbation argument. \medskip

More precisely, let us define a new matrix $\hat{U}=(\hat{U}_{ij})$ by
\begin{equation}\label{U}
\hat{U}_{ij}:=u_{ij}-(\delta_{ij}-\delta_{1i}\delta_{1j})
\end{equation}
such that
\[
\lambda_{1}(\hat{U})\leq \lambda_{1}(U)\qquad \text{ in } \,\Omega,
\]
with equality holds only at $x_0$. We may further order the eigenvalues of $\hat{U}$ by $\lambda_{1}(\hat{U})\geq \cdots \geq \lambda_{n}(\hat{U})$ and then $\lambda_{1}(\hat{U})>\lambda_{2}(\hat{U})$ at the considered point. Hence, the new auxiliary function
\begin{equation*}
\hat{Q}:=\log \lambda_{1}(\hat{U})+\varphi\big(\frac{1}{2}|Du|^{2}\big)+\frac{1}{2}|x|^{2}
\end{equation*}
is twice differentiable pointwise within $\Omega_{+}$. Moreover, the point $x_0$ where the maximum of $Q$ is attained turns out to be a  maximum point for the function $\hat{Q}$. Below it remains to prove that
\begin{equation}\label{global hessian to boundary v2}
\lambda_{1}(\hat{U})(x_0)\leq C_{2}
\end{equation}
for a certain positive constant $C_{2}$ as in \eqref{global hessian to boundary}.

Rotating the coordinates $x_1,\cdots,x_n$, we may assume that $\hat U$ is diagonal at $x_0$. From the maximum principle, it follows that at $x_0$:
\begin{equation*}
0=\hat{Q}_{i}= \frac{\lambda_{1,i}}{\lambda_{1}}+\varphi'u_{ii}u_{i}+x_{i}
\end{equation*}
as well as
\begin{equation*}
\begin{split}
0\geq &\quad \frac{\lambda_{1,ii}}{\lambda_{1}}-\Big(\frac{\lambda_{1,i}}{\lambda_1}\Big)^{2} +\varphi''u^{2}_{ii}u^{2}_{i}+\varphi'\Big(u^2_{ii}+\sum_{j=1}^{n}u_{jii}u_{j}\Big)+1.
\end{split}
\end{equation*}
The following formulas are well-known in \cite{Spr}:
\begin{align*}
\label{}
& \frac{\partial \lambda_{1}}{\partial x_{i}}(x_0)=\frac{\partial u_{11}}{\partial x_{i}}(x_0)\,; \\
\label{}
& \frac{\partial^2 \lambda_{1}}{\partial x_{i}\partial x_{j}}(x_0)=\frac{\partial u_{11}}{\partial x_{i}\partial x_{j}}(x_0)+2\sum_{k=2}^{n}\frac{u_{1ki}u_{1kj}}{\lambda_{1}-\lambda_k+1}\,.
\end{align*}
We therefore obtain
\begin{equation}\label{fisrt order vanish}
\frac{u_{11i}}{u_{11}}=-\varphi'u_{ii}u_{i}-x_{i}
\end{equation}
for each $i=1,\cdots,n$ and 
\begin{equation}\label{max principle}
\begin{split}
0\geq \sum_{i=1}^{n}F^{ii} \Big[\frac{u_{11ii}}{u_{11}} &+2\sum_{k=2}^{n}\frac{u^2_{1ki}}{u_{11}(u_{11}-u_{kk}+1)}-\frac{u^{2}_{11i}}{u^{2}_{11}} \Big] \\
&+\varphi''\sum_{i=1}^{n}F^{ii}u^{2}_{ii}u^{2}_{i}+\varphi'\sum_{i=1}^{n}F^{ii}\Big[u^2_{ii}+\sum_{j=1}^{n}u_{jii}u_{j}\Big]+\mathcal{F}\,.
\end{split}
\end{equation}
Applying the formula \eqref{second derivatives} for $t=1$, we get
\begin{equation}\label{second derivatives with t1}
\sum_{i=1}^{n}F^{ii}u_{11ii}+\sum_{i,j,k,l}F^{ij,kl}u_{1ij}u_{1kl}=\psi_{x_{1}x_{1}}+\psi_{ux_{1}}u_{1}+\psi_{uu}u_{1}^{2}+\psi_{u}u_{11}\geq -Cu_{11}
\end{equation}
provided that $u_{11}\geq 1$ is large enough. While the second term on the left-hand reads
\[
\begin{split}
\sum_{i,j,k,l=1}^{n}F^{ij,kl}u_{1ij}u_{1kl}&= 2\sum_{1\leq l<k\leq n}F^{lk,kl}u^{2}_{1kl}+\sum_{k,l=1}^{n}F^{kk,ll}u_{1kk}u_{1ll}\leq 2\sum_{k=2}^{n}F^{1k,k1}u^{2}_{11k}\,.
\end{split}
\]
It then follows from \eqref{second derivatives with t1} that
\begin{equation}\label{fourth term} 
\sum_{i=1}^{n}F^{ii} \frac{u_{11ii}}{u_{11}} \geq -2\sum_{k=2}^{n}F^{1k,k1}\frac{u^{2}_{11k}}{u_{11}}-C\,.
\end{equation}
Observe that Proposition \ref{gradient estimate} implies that
\begin{equation}\label{third and first terms}
\sum_{i,j=1}^{n}F^{ii}u_{jii}u_{j}  =\langle D_{x}\psi,Du\rangle+D_{u}\psi|Du|^2\geq -C\,.
\end{equation}
Plugging  the inequalities \eqref{fourth term} and \eqref{third and first terms} into \eqref{max principle}, we therefore obtain
\begin{equation}\label{max principle 2}
\begin{split}
C\geq \sum_{k=2}^{n}\Big[-2F^{1k,k1}\frac{u^{2}_{11k}}{u_{11}}&+2F^{11}\frac{u^2_{11k}}{u_{11}(u_{11}-u_{kk}+1)}\Big] -\sum_{i=1}^{n}F^{ii}\frac{u^{2}_{11i}}{u^{2}_{11}} \\
&+\varphi''\sum_{i=1}^{n}F^{ii}u^{2}_{ii}u^{2}_{i}+\varphi'\sum_{i=1}^{n}F^{ii}u^2_{ii}+\mathcal{F}\,.
\end{split}
\end{equation}
In what follows, we divide the proof of inequality \eqref{global hessian to boundary v2} into two separate cases.\medskip

\noindent\textbf{Case 1: Assume $u_{nn}< -\frac{1}{3}\, u_{11}$. } For the second term in \eqref{max principle 2},  it follows from \eqref{phi}, \eqref{fisrt order vanish} and  the Cauchy-Schwarz inequality that
\begin{equation*}
\begin{split}
\sum_{i=1}^{n}F^{ii}\frac{u^{2}_{11i}}{u^{2}_{11}}&\leq 2(\varphi')^2\sum_{i=1}^{n}F^{ii}u^{2}_{ii}u^{2}_{i}+2\sum_{i=1}^{n}F^{ii}x_{i}^{2}\,.
\end{split}
\end{equation*}
Inserting it into \eqref{max principle 2} and coupling with \eqref{phi}, we have
\begin{equation*}
\begin{split}
C\geq &\, \sum_{k=2}^{n}\Big[-2F^{1k,k1}\frac{u^{2}_{11k}}{u_{11}}+2F^{11}\frac{u^2_{11k}}{u_{11}(u_{11}-u_{kk}+1)}\Big] -2\sum_{i=1}^{n}F^{ii}x_{i}^{2} \\
&\qquad +\big(\varphi''-2(\varphi')^2\big)\sum_{i=1}^{n}F^{ii}u^{2}_{ii}u^{2}_{i}+\varphi'\sum_{i=1}^{n}F^{ii}u^2_{ii}+\mathcal{F}\\
\geq &\, \varphi'\sum_{i=1}^{n}F^{ii}u^2_{ii}-2(diam\, \Omega)^{2}\mathcal{F}\,.
\end{split}
\end{equation*}
It then follows from the assumption $u_{nn}< -\frac{1}{3}\, u_{11}$ that
\begin{equation}\label{case 1-1}
C\geq \varphi'\sum_{i=1}^{n}F^{ii}u^2_{ii}-2(diam\, \Omega)^{2}\mathcal{F}\geq \varphi'F^{nn}u^2_{nn}-2(diam\, \Omega)^{2}\mathcal{F}  \geq  \frac{1}{9}\varphi'u_{11}^{2}F^{nn}-2(diam\, \Omega)^{2}\mathcal{F}.
\end{equation}
Given that $u_{11}\geq \cdots \geq u_{nn}$,  we can deduce from a result in Ecker-Huisken \cite{EH} or Spruck \cite{Spr} that $F^{11}\leq \cdots \leq F^{nn}$ and hence $F^{nn}\geq \frac{1}{n}\mathcal{F}$.  Observing that the  Condition {\bf (T)} and the Cauchy-Schwarz inequality guarantee that $\mathcal{F}\geq n\mathscr{T}^{\frac{1}{n}}$, it follows from \eqref{case 1-1} that
\[
C\geq \Big[\frac{1}{9n}\varphi'u_{11}^{2}-2(diam\, \Omega)^{2}\Big]\mathcal{F}\geq  \Big[\frac{u_{11}^{2}}{36nK}-2(diam\, \Omega)^{2}\Big]n\mathscr{T}^{\frac{1}{n}}\,.
\]
This will enable us to derive the desired upper bound of $u_{11}(x_0).$ \medskip

\noindent\textbf{Case 2: Assume $u_{nn}\geq -\frac{1}{3}\, u_{11}$.} Let us introduce an index set
\[
I:=\big\{ i\in \{2,\cdots,n\}\, |\, F^{ii}\geq 3F^{11} \big\}\,.
\]
Then for each $i\notin I$, we have $F^{ii}<3F^{11}$. It follows that
\begin{equation*}
\begin{split}
\sum_{i\notin I}F^{ii}\frac{u^{2}_{11i}}{u^{2}_{11}}\leq & \,2(\varphi')^2\sum_{i\notin I}F^{ii}u^{2}_{ii}u^{2}_{i}+2\sum_{i\notin I}F^{ii}x_{i}^{2} 
\leq  \varphi''\sum_{i\notin I}F^{ii}u^{2}_{ii}u^{2}_{i}+6n(diam\, \Omega)^{2}F^{11}\,.
\end{split}
\end{equation*}
Inserting it into \eqref{max principle 2}, we arrive at
\begin{equation}\label{max principle 3}
\begin{split}
C\geq &\sum_{k=2}^{n}\Big[-2F^{1k,k1}\frac{u^{2}_{11k}}{u_{11}}+2F^{11}\frac{u^2_{11k}}{u_{11}(u_{11}-u_{kk}+1)}\Big] -\sum_{i\in I}F^{ii}\frac{u^{2}_{11i}}{u^{2}_{11}} \\
&+\varphi''\sum_{i\in I}F^{ii}u^{2}_{ii}u^{2}_{i}+\varphi'F^{ii}u^2_{ii}+\mathcal{F}-6n(diam\, \Omega)^{2}F^{11}\\
\geq &\sum_{i \in I}\Big(-2F^{1i,i1}-\frac{F^{ii}}{u_{11}}\Big)\frac{u^{2}_{11i}}{u_{11}}+\varphi'F^{11}u^2_{11}+\mathcal{F}-6n(diam\, \Omega)^{2}F^{11}\,.
\end{split}
\end{equation}

For each $i\in I$, we have $F^{11}\leq \frac{1}{3}F^{ii}$ and $u_{ii}\geq u_{nn}\geq -\frac{1}{3}u_{11}$. It follows from \eqref{structures of F} that 
\begin{equation}\label{Fii and uii}
-2F^{1i,i1}=2\frac{F^{ii}-F^{11}}{u_{11}-u_{ii}}\geq \frac{F^{ii}}{u_{11}} \qquad  \text{for all }\, i\in I\,.
\end{equation}
Therefore, we can infer form \eqref{phi}, \eqref{max principle 3}, and \eqref{Fii and uii} that
\begin{equation}\label{max principle 4}
C\geq \frac{1}{4K}F^{11}u^2_{11}+\mathcal{F}-6n(diam\, \Omega)^{2}F^{11}.
\end{equation}
Without loss of generality, we may assume that $\frac{1}{8K}u_{11}^{2}(x_0)\geq 6n(diam\, \Omega)^{2}$.  Indeed, if this condition fails to hold, our proof is already concluded.  Consequently,  \eqref{max principle 4}  simplifies to 
\begin{equation*}
C\geq \frac{1}{8K}F^{11}u^2_{11}+\mathcal{F}.   
\end{equation*}
Since each $F^{ii}$ is bounded from above, according to Condition {\bf (T)}, $F^{11}$ is bounded from below and we obtain a uniform upper bound for $u_{11}(x_0)$ and thereby complete the proof of proposition.
\end{proof}
\subsection{Boundary second order estimate}
\begin{proposition}\label{boundary C2}
Suppose $\Omega$ is a bounded and strictly $\Gamma$-convex domain in $\mathbb{R}^n$ with smooth boundary. Let $u\in C^{4}(\Omega)\cap C^{1,1}(\overline{\Omega})$ be a $\Gamma$-admissible solution to the Dirichlet problem \eqref{Dirichlet problem}. Then we have
\begin{equation*}
\sup_{\partial\Omega}|D^2u|\leq C\,,
\end{equation*}
where $C$ is a constant depends on $p,$ $\Omega,$  $C_{0}$, $\sup_{\overline{\Omega}\times [-C_{0},0]}\psi$, $\|D_{x}\psi\|_{C^{1}}$  and  $\|D_{u}\psi\|_{C^1}$.
\end{proposition}

For each given boundary point $x\in \partial\Omega$, after a translation and a rotation, we may assume that $x=0$ is  the origin and the boundary $\partial\Omega$ can be represented locally in the same manner as \eqref{coordinate transfor}. Below it suffices to show
\begin{equation}\label{boundary hessian}
|D^2u(0)|\leq C\,.
\end{equation}
\subsubsection{Double tangential estimates}
Differentiating the boundary condition $u(x',\rho(x'))=0$ on $\partial\Omega$ twice along the tangential directions,  we have
\begin{equation}\label{boundary point}
u_{\alpha\beta}(0)=-u_{n}\rho_{\alpha\beta}=-u_{n}\kappa_{\alpha}\delta_{\alpha\beta}\, \qquad   \text{for all } \, 
 \alpha,\beta\leq n-1,
\end{equation}
where $(\kappa_1, \cdots,\kappa_{n-1})$ is the principle curvatures of $\partial\Omega$ at the origin. This implies, using the gradient estimates, that:
\begin{equation}\label{double tangential}
\big|u_{\alpha\beta}(0)\big|\leq C\, \qquad  \text{for all } \, \alpha,\beta\leq n-1\,.
\end{equation}
Notice that the boundedness property of $\kappa_{\alpha}$ is the only aspect we utilize here.

\subsubsection{Tangential to normal estimates}

To establish the estimate
\begin{equation}\label{mix}
|u_{\alpha n}(0)|\leq C\, \qquad   \text{for all } \,  \alpha\leq n-1,
\end{equation}
it is sufficient to proceed as in the proof of Caffarelli-Nirenberg-Spruck \cite[Section 5]{CNS85}, we include a sketch of proof here for completeness.

 Let us define an elliptic operator
\[L=\sum_{i,j=1}^{n}F^{ij}\partial_{i}\partial_{j}.\] 
The following lemma,  which is due to B. Guan \cite{Guan99},  will play a significant role in our proof.
\begin{lemma}
There exist some uniform positive constants $t$, $\delta$ and $\epsilon$ sufficiently small, and an $N$ sufficiently large, such that the function
\begin{equation*}
v=(u-\underline{u})+td-\frac{N}{2}d^{2}
\end{equation*}
satisfies $v\geq 0$ in $\overline{\Omega}_{\delta}$ and
\begin{equation}\label{Lv}
Lv\leq -\epsilon -\epsilon \sum_{i=1}^{n}F^{ii} \qquad \text{ in } \ \Omega_{\delta}\,.
\end{equation}
\end{lemma}

\begin{proof}
First of all, it is easy to see $v=0$ on $\partial\Omega\,\cap\, \Omega_{\delta}$. Then, we can shrink $\delta<2t/N$ such that $v\geq 0$ on $\Omega\,\cap\, \partial\Omega_{\delta}$ after $t$ and $N$ being bounded. Therefore, we infer that $v\geq 0$ in $\overline{\Omega}_{\delta}$.

Because $\underline{u}$ is strictly $p$-plurisubharmonic, there exists a constant $\epsilon_{0}>0$  dependent  on $C_0$ such that 
\[(\underline{u}_{ij}-2\epsilon_0\delta_{ij})\in \mathcal{P}_{p}\,.\]
Notice that 
\[
Ld^2=2dLd+2\sum_{i,j=1}^{n}F^{ij}d_id_j=2dLd+2F^{nn}
\]
since $d_j=0$ for all $j\leq n-1$, and $d_n=1$.  It is also easy to verify that
\[
-C_{d}\Big(1+\sum_{i=1}^{n}F^{ii}\Big)\leq Ld\leq C_{d}\Big(1+\sum_{i=1}^{n}F^{ii}\Big)
\]
for a uniform positive constant $C_d$. 
It follows from \eqref{sum} that 
\begin{equation}\label{Lv II}
\begin{split}
Lv&=\sum_{i,j}F^{ij}(u_{ij}-\underline{u}_{ij})-(Nd-t)Ld-NF^{nn}\\
&\leq \sup_{\overline{\Omega}}\psi(x,u)+C_d(N\delta+t)+\big(C_d(N\delta+t)-2\epsilon_{0}\big)\sum_{i=1}^{n}F^{ii}-NF^{nn}\\
&\leq \sup_{\overline{\Omega}}\psi(x,u)+\epsilon_0-\Big(\frac{\epsilon_{0}}{2}\sum_{i=1}^{n}F^{ii}+NF^{nn}\Big)-\frac{\epsilon_{0}}{2}\sum_{i=1}^{n}F^{ii}
\end{split}
\end{equation}
provided that $C_{d}(N\delta+t)\leq 3C_dt<\epsilon_0$. By the Condition {\bf (T)} and the arithmetic-geometric means inequality, we deduce that
\begin{equation*}
\frac{\epsilon_{0}}{2}\sum_{i=1}^{n}F^{ii}+NF^{nn}
\geq   n\Big(\frac{\epsilon_{0}}{2}\Big)^{1-\frac{1}{n}}N^{\frac{1}{n}}\Big(\prod_{i=1}^{n} F^{ii}\Big)^{\frac{1}{n}} 
\geq n\Big(\frac{\epsilon_{0}}{2}\Big)^{1-\frac{1}{n}}\mathscr{T}^{\frac{1}{n}}N^{\frac{1}{n}}\,.
\end{equation*}
We can choose $N$ large enough to guarantee that 
$$n\Big(\frac{\epsilon_{0}}{2}\Big)^{1-\frac{1}{n}}\mathscr{T}^{\frac{1}{n}}N^{\frac{1}{n}}\geq \frac{3\epsilon_0}{2}+\sup_{\overline{\Omega}}\psi(x,u)\,.$$
By substituting this into \eqref{Lv II} and choosing $\epsilon=\frac{\epsilon_{0}}{2}$, we complete the proof of \eqref{Lv}.
\end{proof}

With this result we will prove the inequality \eqref{mix}. Differentiating the equation $F(D^2u)=\psi(x,u)$ once, we get that for each $k=1,2,\cdots,n$,
\begin{equation*}
-C\Big(1+\sum_{i=1}^{n}F^{ii}\Big)\leq L(u_k-\underline{u}_k) \leq C\Big(1+\sum_{i=1}^{n}F^{ii}\Big)\,.
\end{equation*}
Consider the auxiliary function given by
\begin{equation*}
Q=\pm (u_{\alpha}-\underline{u}_{\alpha})+Av+B\rho^{2}\, \qquad   \text{for all } \alpha< n\,,
\end{equation*}
where $A$ and $B$ are constants with  $A\gg B \gg 1$,  ensuring that $Q\geq 0$ on the boundary $\partial\Omega_{\delta}$ and $LQ\leq 0$ in $\Omega_{\delta}$. According to the maximum principle, it can be deduced that $Q\geq 0$  across the entire domain $\Omega_{\delta}$. Given that $Q(x_0)=0$,  the Hopf Lemma \cite[Lemma 6.4.2]{E10} implies
\[
Q_n(x_0)\leq 0\,.
\]
This provides us with the desired estimate \eqref{mix}. 

\subsubsection{Double normal estimates}

Below we will derive the following  double normal estimate
\begin{equation}\label{double normal}
|u_{nn}(0)|\leq C\,.
\end{equation}
Prior to proving this, let us revisit some valuable concepts from matrix theory: For each $A=(a_{ij})\in \mathrm{Sym}^{2}(\mathbb{R}^{n})$ with eigenvalues $\lambda_{i}(A)$ arranged in ascending order, i.e. $\lambda_{1}(A)\leq \cdots \leq \lambda_{n}(A)$, we denote the eigenvalues of $\tilde{A}=(a_{\alpha\beta})_{1\leq \alpha,\beta\leq
 n-1}$ by $\lambda'_{\alpha}(\tilde{A})$. It follows from Cauchy's interlace inequality \cite{H04} and \cite[Lemma 1.2]{CNS85} that when $|a_{nn}|\rightarrow \infty$
\begin{equation}\label{matrix theory}
\begin{matrix}
\begin{aligned}
&\lambda_{\alpha}(A)\leq \lambda'_{\alpha}(\tilde A)\leq \lambda_{\alpha+1}(A) \\[2mm]
&\lambda_{\alpha+1}(A) =\lambda'_{\alpha}(\tilde A)+o(1) \\[2mm]
&a_{nn} \leq \lambda_{n}(A)\leq  a_{nn}+O(1)
\end{aligned} 
\end{matrix} 
\end{equation}
for all $1\leq \alpha\leq n-1$, where the $o(1)$ and $O(1)$ are uniform terms. 

Next, we are proceed to prove \eqref{double normal}. Without loss of generality, we assume that $u_{nn}(0)\gg 1$ is sufficiently large; otherwise, there is nothing to prove. We claim that there are two uniform positive constants, $\hat{\delta}$ and $R_{0}$, such that for each $R\geq R_{0}$ and for each $x\in \partial\Omega$, we have
\begin{equation*}
\big(\lambda'(\tilde{U}),R\big)\in \Gamma  \quad  \textrm{and} \quad f\big(\lambda'(\tilde{U}),R\big)\geq \psi(x,u)+\hat{\delta}.
\end{equation*}
To achieve this, we shall follow an approach proposed by Trudinger \cite{Tru95} and set
\begin{equation*}
\delta_{R}:=\min_{x\in \partial\Omega}\big[f\big(\lambda'(\tilde{U}),R\big)-\psi(x,u)\big]\,, \qquad \delta_{\infty}:=\liminf_{R\rightarrow \infty}\delta_{R}\,.
\end{equation*}
Subsequently,  we are left with showing
\begin{equation}\label{delta}
\delta_{\infty}\geq \hat{\delta}
\end{equation}
for some uniform constant $\hat{\delta}>0.$

We can assume that $\delta_{\infty}$ is finite; otherwise, our proof is complete. Suppose that $\delta_{\infty}$ is attained at a point $x_{0}\in \partial\Omega$. If necessary, after performing a rotation, we select a local orthonormal frame $(e_{1},\cdots,e_{n})$ such that matrix $\tilde{U}$ is diagonal at the point $x_0$. Let us denote
\[
\Gamma_{\infty}:=\{\lambda'\in \mathbb{R}^{n-1}\,|\, (\lambda',R)\in \Gamma \, \textrm{ for some } \, R\}.
\]
We will now proceed to divide the proof of \eqref{delta} into two distinct cases. \medskip

\noindent \textbf{Case 1:} Assume that $\underset{\lambda_{n}\rightarrow \infty}{\lim}f(\lambda',\lambda_{n})=\infty$ for each $\lambda'\in \Gamma_{\infty}$. Based on \eqref{double tangential} and \eqref{mix}, it can be deduced that   $\lambda'(\tilde{U}(x_{0}))$ belongs to a compact set $ \mathcal{D}$ which is a subset of $ \Gamma_{\infty}$. 
Due to the compactness, there exist two constants, $\delta',R'>0$, which depend on $\lambda'(\tilde{U}(x_{0}))$, such that
\begin{equation*}
f\big(\lambda'(\tilde{U}(x_{0})),R\big)\geq \psi(x_{0},u)+2\delta' \qquad  \text{for all } \, R\geq R'.
\end{equation*}
By continuity, there exists an open cone $\mathcal{D}'$
with $\mathcal{D}\subset \mathcal{D}' \subset \Gamma_{\infty}$ such that
\begin{equation}\label{continuity}
f(\lambda',R)\geq \psi(x_{0},u)+\delta'  \qquad \text{for all } \, \lambda'\in \mathcal{D}'\, \textrm{ and }\, R\geq R'.
\end{equation}
Applying \eqref{matrix theory} to the matrix $D^2u(x_{0})$, there exists a sufficiently large constant $R''\geq R'$ such that if $u_{nn}(x_{0})\geq R''$, we have
\begin{equation}\label{max eigen}
\lambda_{n}(U(x_{0}))\geq u_{nn}(x_{0})\geq R''\geq R'.
\end{equation}
Shrinking $\mathcal{D}'$ if necessary such that
\begin{equation}\label{other eigen}
\big(\lambda_{1}(U(x_{0})),\cdots, \lambda_{n-1}(U(x_{0}))\big)\in \mathcal{D}'.
\end{equation}
From \eqref{continuity}, \eqref{max eigen}, and \eqref{other eigen}, we deduce that
\begin{equation*}
F(D^2u)(x_0)\geq \psi(x_{0},u)+\delta' ,
\end{equation*}
leading to a contradiction with \eqref{Dirichlet problem}. Consequently, we establish \eqref{delta} by setting $\hat{\delta}= \delta'$.
\medskip

\noindent \textbf{Case 2:} Assume that $\underset{\lambda_{n}\rightarrow \infty}{\lim}f(\lambda',\lambda_{n})<\infty$ for each $\lambda'\in \Gamma_{\infty}$. We define a new function
$$\tilde{F}(E):=\lim_{R\rightarrow\infty}f(\lambda'(E),R)$$
on the set $S:=\{E\in \mathrm{Sym}^{2}(\mathbb{R}^{n-1})\,|\, \lambda'(E)\in \Gamma_{\infty} \}$. Observe that $\tilde{F}$ is both concave and finite, given that the operator $F$ is concave and continuous. Consequently, there exists a symmetric $(n-1)\times (n-1)$ matrix $(\tilde{F}^{\alpha\beta})$ such that
\begin{equation}\label{concave}
\sum_{\alpha,\beta=1}^{n-1}\tilde{F}^{\alpha\beta}(\tilde{U})\big(E_{\alpha\beta}-\tilde{U}_{\alpha\beta}\big)\geq \tilde{F}(E)-\tilde{F}(\tilde{U}), \qquad \text{ for all}\, E\in S.
\end{equation}
Analogous to \eqref{boundary point}, one can derive 
\[
\sum_{\alpha,\beta=1}^{n-1}\big(\tilde{\underline{U}}_{\alpha\beta}-\tilde{U}_{\alpha\beta}\big)=(u-\underline{u})_{x_{n}}\sum_{\alpha,\beta=1}^{n-1}\kappa_{\alpha}\delta_{\alpha\beta}
\]
at $x_0$. Substituting this into \eqref{concave} results in
\begin{equation}\label{concave1}
\begin{split}
(u-\underline{u})_{x_{n}}\sum_{\alpha=1}^{n-1}\tilde{F}^{\alpha\alpha}\kappa_{\alpha} 
\geq & \, \tilde{F}[\tilde{\underline{U}}]-\tilde{F}[\tilde{U}]\\ = &\tilde{F}[\tilde{\underline{U}}]-\psi(x_0,u)-\delta_{\infty} 
=   \, \tilde{F}[\tilde{\underline{U}}]-\psi(x_0,\underline{u})-\delta_{\infty}\\
\geq & \, \tilde{F}[\tilde{\underline{U}}]-F[\underline{U}]-\delta_{\infty} \\
\geq & 2\epsilon_{0}-\delta_{\infty}\,.
\end{split}
\end{equation}
Here the second inequality follows from the condition $u=\underline{u}$ on $\partial\Omega$. Additionally,  the constant $\epsilon_{0}:=\frac{1}{2}\min_{\partial\Omega}(\tilde{F}[\tilde{\underline{U}}]-F[\underline{U}])$ is strictly positive due to the ellipticity of $F$.

Without loss of generality, we may assume
\begin{equation}\label{case 2}
(u-\underline{u})_{x_{n}}(x_{0})\sum_{\alpha=1}^{n-1}\tilde{F}^{\alpha\alpha}\kappa_{\alpha}(x_0) \geq\epsilon_{0}.
\end{equation}
 If this were not the case,  we can deduce from \eqref{concave1} that $\delta_{\infty} \geq\epsilon_{0}>0$, which would suffice to conclude the proof for this specific scenario. Henceforth, let us denote 
 \[\eta(x_0)=\sum_{\alpha=1}^{n-1}\tilde{F}^{\alpha\alpha}\kappa_{\alpha}(x_0).\]
Notice that $(u-\underline{u})_{x_{n}}(x_{0})\geq0$ and hence also strictly positive by virtue of \eqref{case 2}. It follows that
\begin{equation}\label{lower bound of eta}
\eta(x_{0})\geq\frac{\epsilon_{0}}{2(u-\underline{u})_{x_{n}}(x_{0})}\geq2\epsilon_1\epsilon_{0} 
\end{equation}
for some uniform constant $\epsilon_1>0$. We may assume $\eta\geq\epsilon_1\epsilon_{0}$ in a small neighborhood $\Omega_{\epsilon}:=\Omega\cap B_{\epsilon}(x_{0})$ when $\epsilon$ is sufficiently small.

Let us define a function $\Phi$ in $\Omega_{\epsilon}$ by
\begin{equation*}
\begin{split}
\Phi=&-\eta(u-\underline{u})_{x_{n}}+\sum_{\alpha,\beta=1}^{n-1}\tilde{F}^{\alpha\beta}\big(\tilde{\underline{U}}_{\alpha\beta}-\tilde{U}_{\alpha\beta}(x_{0})\big). \\
\end{split}
\end{equation*}
A straightforward computation gives
\begin{equation*}
-\eta\,(u-\underline{u})_{x_{n}}=\sum_{\alpha,\beta=1}^{n-1}\tilde{F}^{\alpha\beta}\big(\tilde{U}_{\alpha\beta}-\tilde{\underline{U}}_{\alpha\beta}\big)\,.
\end{equation*}
This, together with \eqref{concave}, yield that
\begin{equation*}
\begin{split}
\Phi&=\sum_{\alpha,\beta=1}^{n-1}\tilde{F}^{\alpha\beta}\big(\tilde{U}_{\alpha\beta}-\tilde{U}_{\alpha\beta}(x_{0})\big)\geq \tilde{F}[\tilde{U}]- \tilde{F}[\tilde{U}(x_{0})]= \tilde{F}(\tilde{U})-\psi(x_{0},u)- \delta_{\infty}\,.\\
\end{split}
\end{equation*}
It follows that $\Phi(x_{0})=0$ and $\Phi\geq0$ on $\partial\Omega$. Let us consider a barrier function
\begin{equation*}
\Psi:=\Phi+N_2 v+N_1|x|^{2} \qquad  \text{ in } \, \Omega_{\epsilon},
\end{equation*}
where $N_{1},N_{2}$ are certain large positive constants. One can verify that 
\[\underset{x\rightarrow\partial\Omega_{\epsilon}}{\lim\inf}\, \Psi(x)\geq 0,\qquad L\Psi\geq0 \, \text{ in } \, \Omega_{\epsilon}\]
as long as $N_2\gg N_1\geq C\epsilon^{2}$.  According to Hopf's lemma, we know that $\Phi_{x_n}(x_{0})\geq-C$. This, in conjunction with \eqref{lower bound of eta},  provides a uniform upper bound for  $u_{x_{n}x_{n}}(x_{0})$.

Now, given that all the eigenvalues of $U(x_0)$ are bounded from above, we can conclude that $\lambda(U)(x_{0})$  lies in a compact subset of $\Gamma$. Using the ellipticity of $f$ again, we see that
\begin{equation*}
\delta_{\infty} \geq\delta_{R_{0}}=f(\lambda'(\tilde{U}
(x_{0})),R_{0})-\psi(x_{0},u)>0
\end{equation*}
when $R_{0}$ is sufficiently large. Consequently, we have finished the proof of \eqref{delta} if we choose $\hat{\delta}=\min\{ \epsilon_{0},\delta_{R_{0}}\}$ for a large constant $R_{0}$.

\begin{remark}\label{higher order}
If $\|\psi\|_{C^{1,1}}$ is uniform 
 bounded, one can apply the Evans-Krylov local argument to conclude the interior $C^{2,\alpha}$ estimate. Then, for every subdomain $\Omega'\subset \Omega$ and for every $\alpha\in (0,1)$, there exists a uniform constant $C>0$ such that
\begin{equation*}
\|u\|_{C^{2,\alpha}(\overline{\Omega'})}\leq C. 
\end{equation*}
The interior $C^{\infty}$ estimate can be also derived using the Schauder theory and a standard iterating argument. We therefore obtain $u\in  C^{\infty}(\Omega) \cap C^{1,1}(\overline{\Omega})$, which completes the proof of Theorem \ref{existence theorem}.
\end{remark}

\section{The Dirichlet eigenvalue problem}
In this section, drawing inspiration from the Lions' approximation strategy, we will complete the proof of Theorem \ref{eigenvalue theorem} in the next three steps: 
\begin{itemize}\itemsep 0.5em 
\item[Step 1:]  We solve a family of Dirichlet problems \eqref{family} that blow up as the parameter approaches a threshold value $\lambda_{1}$.
\item[Step 2:] After a normalization, we apply the previously established estimates and the Arzela-Ascoli theorem to exact a subsequnce that converges to a solution $(\lambda_{1},u_1)$ of \eqref{eigenvalue problem}.
\item[Step 3:] We demonstrate that the solutions of \eqref{eigenvalue problem} are unique, thereby confirming that $(\lambda_{1},u_1)$ is precisely the solution we are seeking.
\end{itemize}

\subsection{Lions'  approximation strategy}\label{Lions}
Following the procedure of Lions \cite{L85}, we may consider a family of Dirichlet problems as follows:
\begin{equation}\label{family}
\left\{\begin{array}{ll}
\begin{matrix}
\begin{aligned}
F(D^2u) =&\, 1-\lambda u  \\
 u= &\, 0
\end{aligned} 
&
\begin{aligned}
\quad &\mbox{in } \Omega,  \\
\quad &\mbox{on } \partial\Omega.
\end{aligned} 
\end{matrix}
\end{array}\right.  
\end{equation}
where $\lambda$ is a parameter. For every $\lambda\geq 0$, we denote the solution of \eqref{family}, if it exists, by $u_{\lambda}$ and define
\begin{equation}\label{Lambda}
I=\big\{\lambda\geq 0: \eqref{family} \textrm{ admits a solution } u_{\lambda}\in C^{1,1}(\overline{\Omega}) \big\}\,, \qquad \Lambda_{1}=\sup_{\lambda\in  I}\lambda\,.
\end{equation}

We begin by noting that in the case of $\lambda=0$, Equation \eqref{family} is non-degenerate. This allows us to adapt the approach of Theorem \ref{existence theorem} to produce a smooth $\Gamma$-admissible solution $u_0$. Therefore, the set $I$ is nonempty and $\Lambda_{1}\geq 0$.

Next, we show that $\Lambda_{1}$ possesses explicit numerical upper or lower bounds.
\begin{lemma}\label{lambda1}
Let $\mu_{1}$ be the first non-zero eigenvalue of the Laplace operator on $\Omega$. We have
\begin{equation}\label{upper lower}
\|u_0\|^{-1}_{C^{0}(\overline{\Omega})}\leq \Lambda_{1}\leq \frac{\mu_{1}}{n}\,.
\end{equation}
\end{lemma}
\begin{proof}
Assuming that the pair $(\lambda,u_{\lambda})$ represents a solution to the Dirichlet problem \eqref{family}, it follows from \eqref{urbas} that
\begin{equation*}
\frac{1}{n}\,\sigma_1(D^2 u_{\lambda})\geq F(D^2u_{\lambda})=1-\lambda u_{\lambda}> -\lambda u_{\lambda}\,.
\end{equation*}
This shows the second inequality of \eqref{upper lower}.

We now proceed to verify the first inequality of \eqref{upper lower}. For an arbitrary positive constant $\lambda<\|u_0\|^{-1}_{C^{0}(\overline{\Omega})}$, we set $C_{\lambda}=(1-\lambda\|u_0\|_{C^{0}(\overline{\Omega})})^{-1}$. Then, according to 1-homogeneity of $F$, we infer that
\begin{equation*}
F\big(D^2 (C_{\lambda}u_{0}) \big)=C_{\lambda}\geq 1-\lambda C_{\lambda}u_{0}\geq 1\,.
\end{equation*}
That is, $C_{\lambda}u_0$ is a subsolution of the Dirichlet problem \eqref{family}. Observe that $u_0$ is a supersolution of   \eqref{family}. In view of Theorem \ref{existence theorem}, the Dirichlet problem \eqref{family} admits a solution $u_{\lambda}\in C^{\infty}(\Omega)\cap C^{1,1}(\overline{\Omega})$. This proves that each $\lambda<\|u_0\|^{-1}_{C^{0}(\overline{\Omega})}$ satisfies the inequality $\lambda\leq \Lambda_{1}$.  Thus, we have established the first inequality presented in \eqref{upper lower}.
\end{proof}

\begin{lemma}\label{short time existence}  For every $\lambda\in [0,\Lambda_{1})$, the Dirichlet problem \eqref{family} admits a $\Gamma$-admissible solution $u_{\lambda}$, i.e. $[0,\Lambda_{1})\subset I$.
\end{lemma}
\begin{proof} Since $\lambda<\Lambda_{1}$  and we have already established that $u_0$ is a supersolution of  the Dirichlet problem \eqref{family}, it is sufficient to construct a subsolution. 

From the definition of $\Lambda_{1}$ as a supremum, there exists a constant $\tilde{\lambda}$ with $\lambda<\tilde{\lambda}<\Lambda_{1}$ such that $\tilde{\lambda}\in I$. It follows that $u_{\tilde{\lambda}}$ is a desired subsolution of \eqref{family}. Indeed, since $u_{\tilde{\lambda}}\leq 0$ in $\Omega$ 
\begin{equation*}
F(D^2u_{\tilde{\lambda}})=1-\tilde{\lambda} u_{\tilde{\lambda}} \geq 1-\lambda u_{\tilde{\lambda}}\,.  
\end{equation*}
This ends our proof of lemma.
\end{proof}

\begin{lemma}\label{step 4}
$\|u_{\lambda}\|_{C^{0}(\overline{\Omega})}$ tends to infinity as $\lambda\nearrow \Lambda_{1}$.   
\end{lemma}
\begin{proof}
Suppose the contrary, then there exists a sequence of numbers $\mu_{j}\in \Lambda$ such that $\mu_{j}\rightarrow \Lambda_{1}$ with $\|u_{\mu_{j}}\|_{C^{0}(\overline{\Omega})}\leq N<+\infty$ for a constant $N$ independent of $l$. Thus, we have $-N\leq u_{\mu_{j}}\leq 0$ on $\overline{\Omega}$ and whence
\begin{equation*}
 F(D^2u_{0})=1\leq F(D^2u_{\mu_{j}})=1-\mu_{j}u_{\mu_{j}}\leq 1+N\Lambda_{1}= F\big((1+N\Lambda_{1})D^2u_{0}\big)\,.
\end{equation*}
Applying the comparison principle, it follows that 
\begin{equation*}
(1+N\Lambda_{1})u_{0}\leq u_{\mu_{j}}\leq u_{0}<0 \qquad \textrm{ in } \Omega.
\end{equation*}

From Propositions \ref{gradient estimate}, \ref{int hessian estimate} and \ref{boundary C2}, we deduce that $\|u_{\mu_{j}}\|_{C^2(\overline{\Omega})}\leq C$ for a uniform constant $C$. Furthermore, by applying the standard Evans-Krylov theory for non-degenerate uniformly elliptic equations, we also obtain local bounds in $C^{k}(\Omega)$ for all $k\geq 1 $. Then,  by extracting a subsequence if necessary, we may assume $\{u_{\mu_{j}}\}$ converges to some function $u_{\mu_{\infty}}\in C^{1,1}(\overline{\Omega})\cap C^{\infty}(\Omega)$ as $\mu_{j}\nearrow \lambda_{1}$.

As in the proof of Lemma \ref{lambda1}, for every $0<\delta<\|u_{\mu_{\infty}}\|^{-1}_{C^{0}(\overline{\Omega})}$, we define a constant $C_{\delta}=(1-\delta\|u_{\mu_{\infty}}\|_{C^{0}(\overline{\Omega})})^{-1}$ to ensure that
\begin{equation*}
F\big(D^2(C_{\delta}u_{\mu_{\infty}})\big)= C_{\delta}(1-\Lambda_{1} u_{\mu_{\infty}})\geq 1-(\delta+\Lambda_{1})C_{\delta}u_{\mu_{\infty}},
\end{equation*}
i.e. $C_{\delta}u_{\mu_{\infty}}$ serves as a subsolution of \eqref{family} with $\lambda=\delta+\Lambda_{1}$. 

Observe that $u_{0}$ remains a supersolution of \eqref{family}. Thanks to Theorem  \ref{existence theorem}, then the Dirichlet problem \eqref{family} admits a global $C^{1,1}$ solution for every $\Lambda_{1}\leq  \lambda<\Lambda_{1}+\delta$, which contradicts our definition of $\Lambda_{1}$. This asserts the lemma.
\end{proof}

\subsection{Proof of Theorem \ref{eigenvalue theorem}}
For every constant $0\leq \mu<\Lambda_{1}$, we consider the following normalization for $u_{\mu}$:
\begin{equation*}
v_{\mu}:=\frac{u_{\mu}}{\|u_{\mu}\|_{C^{0}(\overline{\Omega})}}\,.
\end{equation*}
It can be easily seen that $v_{\mu}$  is a solution of the Dirichlet problem:
\begin{equation*}
\left\{\begin{array}{ll}
\begin{matrix}
\begin{aligned}
F(D^2v_{\mu}) =&\, \|u_{\mu}\|^{-1}_{C^{0}(\overline{\Omega})}-\mu v_{\mu}  \\
 v_{\mu}= &\, 0   \\
\|v_{\mu}\|_{C^{0}(\overline{\Omega})}=&\, 1.
\end{aligned} 
&
\begin{aligned}
\quad &\mbox{in } \Omega,  \\
\quad &\mbox{on } \partial\Omega.
\end{aligned} 
\end{matrix}
\end{array}\right.
\end{equation*}
In accordance with Lemma \ref{step 4}, there exists a constant $\nu_{1}$ belonging to the interval $(0,\Lambda_{1})$ such that for every $\mu$ in the range $(\nu_{1},\Lambda_{1})$,  the norm $\|u_{\mu}\|_{C^{0}(\overline{\Omega})}$ is greater than $1$. Besides, since $-1\leq v_{\mu}<0$ in $\Omega$, it follows that 
\[
\|u_{\mu}\|^{-1}_{C^{0}(\overline{\Omega})}-\mu v_{\mu}\leq 1+\mu\leq 1+\Lambda_{1} \qquad  \text{for all }\, \mu\in (\nu_{1},\Lambda_{1})\,.
\]
Applying the comparison principle, we immediately infer that \begin{equation}\label{lower barrier}
(1+\Lambda_{1})u_{0}\leq v_{\mu}\leq 0 \qquad  \text{for all }\, \mu\in (\nu_{1},\Lambda_{1})
\end{equation}
which yields a uniform upper bound on $\|D v_{\mu}\|_{C^{0}(\partial\Omega)}$ independent of $\mu$. Then by Propositions \ref{gradient estimate}, \ref{int hessian estimate} and \ref{boundary C2}, we see that 
\begin{equation*}
\|v_{\mu}\|_{C^2(\overline{\Omega})}\leq C, \qquad  \text{for all }\, \mu\in (\nu_{1},\Lambda_{1})
\end{equation*}
for a constant $C$ independent of $\mu$. 

By virtue of the Arzela-Ascoli theorem, we can extract  a subsequence $\{\mu_{j}\}$ within the interval $ (\nu_{1},\Lambda_{1}) $ that converges to $\Lambda_{1}$, such that $\{v_{\mu_j}\}$ converges uniformly to specific function $u_{1}\in C^{0}(\overline{\Omega})$. Furthermore, $u_1$ is indeed a weak solution to the following Dirichlet problem: 
\begin{equation}\label{limit equation}
\left\{\begin{array}{ll}
\begin{matrix}
\begin{aligned}
F(D^2u_{1}) =&\, -\Lambda_{1} u_{1}  \\
 u_{1}= &\, 0   \\
\|u_{1}\|_{C^{0}(\overline{\Omega})}=&\, 1.
\end{aligned} 
&
\begin{aligned}
\quad &\mbox{in } \Omega,  \\
\quad &\mbox{on } \partial\Omega.
\end{aligned} 
\end{matrix}
\end{array}\right.
\end{equation}
Below we will exploit the higher order regularity of $u_{1}$, which can be derived through a barrier argument. 
\begin{lemma}
Let $u_{1}$ be a continuous solution of \eqref{limit equation}. Then we have $u_1\in C^{\infty}(\Omega)\cap C^{1,1}(\overline{\Omega})$.
\end{lemma}
\begin{proof}
Given that $\|u_{1}\|_{C^{0}(\overline{\Omega})}=1$ and $u_{1}=0$ on $\partial\Omega$, 
 there exists an interior point $x_0\in\Omega$ such that $u_{1}(x_0)=-1$. Observe that $v_{\mu_j}\rightarrow u_{1}$ implies that 
 \[-v_{\mu_j}(x_0)\geq -u_{1}(x_0)-\frac{1}{3}=\frac{2}{3}\,.\] 
Due to the continuity, there exists a small ball $B_{2r_0}=B(x_0,2r_0)$ centered at $x_0$ within $\Omega$, for some constant $r_0>0$, such that 
\begin{equation*}
v_{\mu_j}\leq -\frac{1}{3}\qquad \textrm{ on }  B_{2r_0}
\end{equation*}
provided $j\in \mathbb{N}$ is large enough. Then, it follows from the \eqref{urbas} that
\[
\Delta v_{\mu_j}\geq nF(D^2v_{\mu_j})\geq -n\mu_{j}v_{\mu_j}\geq \frac{n\nu_{1}}{3} \qquad \textrm{ on }  B_{2r_0}.
\]
Select a smooth function $h:[0,+\infty)\rightarrow \mathbb{R}$ such that 
\begin{equation*}
0\leq h \leq 1, \qquad |h'|+|h''|\leq N 
\end{equation*}
for some constant $N $ that depends only on $r_0$ and  
\begin{equation*}
h(r)= 
\left\{\begin{array}{ll}
1,\quad&r\leq r_0\,; \\[1mm]
0,& r\geq 2r_0\,.
\end{array}\right.
\end{equation*}
Let $v$ be a smooth solution to the following Dirichlet problem:
\begin{equation*}
\left\{\begin{array}{ll}
\begin{matrix}
\begin{aligned}
\Delta v=&\, \frac{n\nu_{1}}{3}h\big(|x-x_0|\big)  \\
 v= &\, 0
\end{aligned} 
&
\begin{aligned}
\quad &\mbox{in } \Omega,  \\
\quad &\mbox{on } \partial\Omega.
\end{aligned} 
\end{matrix}
\end{array}\right.  
\end{equation*}
It is evident that $\sigma_{1}(D^2v)\leq \sigma_{1}(D^2v_{\mu_j})$ in $\Omega$ and $v=v_{\mu_j}$ on $\partial\Omega$. In light of the comparison principle and a similar argument as presented in 
\eqref{lower barrier}, it follows that 
\begin{equation*}\label{uniform bounded}
(1+\Lambda_{1})u_0\leq v_{\mu_j}\leq v \qquad \textrm{ in } \Omega.
\end{equation*}
Therefore, $\|u_{\mu_j}\|_{C^{0}(\overline{\Omega})}$ is uniform bounded provided $j\in \mathbb{N}$ is large enough. With the help of Theorem \ref{existence theorem}, there exists a uniform constant $C$ such that
\begin{equation}\label{c11 estimates}
\|v_{\mu_j}\|_{C^{1,1}(\overline{\Omega})}\leq C\,.  
\end{equation}

Observe that $v_{\mu_j}$ is a solution to the following equation 
\begin{equation*}
F(D^2v_{\mu_{j}})=\|u_{\mu_{j}}\|^{-1}_{C^{0}(\overline{\Omega})}-\mu_{j} v_{\mu_{j}}=:h_{j} \quad \textrm{in } \Omega\,.
\end{equation*}
Considering the facts that $\mu_{j}\in (\nu_{1},\Lambda_{1}) $ and \eqref{c11 estimates}, we observe that $\|h_{j}\|_{C^{1,1}(\overline{\Omega})}\leq C'$ for a uniform constant $C'$ that is independent of $j$. Performing a similar argument as presented in Remark \ref{higher order}, there exists a uniform constant $C''$ that is independent of $j$ such that
\begin{equation*}
\|v_{\mu_j}\|_{C^{2,\alpha}(\Omega)}\leq C''\,.  
\end{equation*}
As a consequence, we conclude that $u_1\in C^{\infty}(\Omega)\cap C^{1,1}(\overline{\Omega})$ and $(\Lambda_{1},u_1)$ is a solution of \eqref{eigenvalue problem}.
\end{proof}
We are now finally in a position to prove the uniqueness of the solution to the Dirichlet problem \eqref{eigenvalue problem}.
\begin{lemma}
Let $(\Lambda_2,u_2)\in [0,+\infty)\times (C^{\infty}(\Omega)\cap C^{1,1}(\overline{\Omega}))$ be another solution of the eigenvalue problem \eqref{eigenvalue problem}. Then, we have $\Lambda_{2}=\Lambda_{1}$ and $u_2=cu_{1}$ for some constant $c>0$.
\end{lemma}
\begin{proof} This proof follows a standard argument, which can be found in \cite{ZZ25}, we include it here for completeness.

\textbf{Step 1. $\Lambda_1$ equals to $\Lambda_2$.}  Let us define the elliptic operator
\[
L_1:=\sum_{i,j=1}^{n}F^{ij}(D^2u_1)\partial_{i}\partial_{j}\,.
\]
Since $L_1(u_1)=F(D^2u_1)=-\Lambda_{1}u_1$ and $u_1$ vanishes on $\partial\Omega$, it follows from  \cite[p. 361]{E10} that 
\begin{equation}\label{Evans}
\Lambda_{1}=\mathrm{Re}(\Lambda_1)\geq \Lambda_1(L_1)\,.
\end{equation}
On the other hand, by the characterization of the principle eigenvalue in Berestycki-Nirenberg-Varadhan \cite{BNV94}, we have
\[
\Lambda_{1}(L_1)=\sup\{\lambda\geq 0\,:\,\exists \varphi<0,\, L_1(\varphi)+\lambda\varphi\geq 0\}\,,
\]
where $\varphi\in C^{1,1}(\overline{\Omega})$ and $\varphi=0$ on $\partial\Omega$. Owing to the concavity and homogeneity of the operator $F$, we obtain 
\begin{equation}\label{conc homo}
L_1(u_2)\geq F(D^2u_1+D^2u_2)-F (D^2u_1)\geq F (D^2u_2)=-\Lambda_2u_2.
\end{equation}
Combining this with inequality \eqref{Evans}, we conclude that $\Lambda_2\leq \Lambda_1(L_1)\leq \Lambda_1$. A symmetric argument yields $\Lambda_1\leq \Lambda_2$.    Therefore, $\Lambda_{1}=\Lambda_2.$

\textbf{Step 2. $u_2$ is a positive constant multiple of $u_{1}$.} %We follow the argument of Badiane and Zeriahi \cite{BZ23}, which is ispired by the well-known result of Berestycki-Nirenberg-Varadhan \cite{BNV94}. 
We  claim that there exists $t>0$ such that 
\[0\leq  t(-u_1)\leq (-u_2) \,  \textrm{ in } \, \overline{\Omega}.\] 
To prove this, first note that since $u_2$ is subharmonic, Hopf’s Lemma implies that $\frac{\partial u_2}{\partial n} > 0$ on $\partial\Omega$, where $n$ denotes the outer unit normal vector field on $\partial\Omega$. Let $d$ be the distance function to $\partial\Omega$. Then there exists a uniform constant $c_0 > 0$ such that $d \leq c_0 (-u_2)$ in $\overline{\Omega}$. On the other hand, let $\rho$ be a smooth defining function for $\Omega$ such that $(-\rho) \leq c_1 d$ for some constant $c_1 > 0$. Since both $u_1$ and $\rho$ vanish on $\partial\Omega$, we may write $u_1 = h \rho$ for some smooth function $h$ defined near $\partial\Omega$. It follows that there exists a constant $c_2 > 0$ such that
\[
0\leq (-u_{1})\leq c_{2}(-\rho)\leq c_1c_2d\leq c_0c_1c_2(-u_2).
\]
Thus, the claim holds with $t=(c_0c_1c_2)^{-1}$.

The remainder of the proof follows a standard comparative argument as in \cite{L85}. Define
\[t_0=\sup\{t>0\,:\, 0\leq t(-u_1)\leq (-u_2) \, \textrm{ in } \overline{\Omega}\}.\]
Denote $w=u_2-t_0u_1$ for brevity. It follows from \eqref{conc homo} that $L_1(w)\geq -\lambda_1w$ in $\Omega$. Notice that $w\leq 0$ in $\Omega$ and $w=0$ on $\partial{\Omega}$. There are two cases to consider:
\begin{itemize}\itemsep 0.5em
\item[(i)] Assume that $w=0$ at some point $x_0\in\Omega$.  Since $x_0$ is a maximum point of $w$, the strong maximum principle forces that $w\equiv 0$ in $\overline{\Omega}$. 
\item[(ii)]  If instead $w<0$ in $\Omega$, then $L_1(w)>0$. Applying  Hopf’s Lemma again, there exist a constant $t_1>0$ such that 
\[0\leq t_1(-u_1)\leq (-w)=t_0u_1-u_2,\]
which implies $(t_0+t_1)(-u_1)\leq (-u_2)$ in $\overline{\Omega}$. This  contradicts the definition of $t_0$.
\end{itemize}
In both cases, we conclude that  $u_2=t_0u_1$, which completes the proof.
\end{proof}

\section{Application to the bifurcation-type theory}\label{Application to the bifurcation theorem}
The  bifurcation-type theory for the Monge–Ampère operator was originally established by Lions \cite[Corollary 2]{L85}, in a form analogous to corresponding results for the first eigenvalue of linear elliptic operators. Inspired by this work, we prove in this manuscript that the first eigenvalue of $F$ serves as a bifurcation point for the Dirichlet problem \eqref{Dirichlet problem}.
\begin{proposition}\label{bif}
Let $\psi(x,s)$ be a smooth positive function on $\overline{\Omega}\times (-\infty,0]$ such that 
\[\psi_{s}\geq -\Lambda_{0} > -\Lambda_{1},\]
where $\Lambda_0$ is a given non-negative constant and $\Lambda_1$ denotes the first eigenvalue of $F$ on $\Omega$.  Then the Dirichlet problem \eqref{Dirichlet problem} admits a $\Gamma$-admissible solution $u\in C^{\infty}(\overline{\Omega})$.
\end{proposition}
\begin{proof} 
The proof is based on the approach of \cite{CLM23}, which was originally formulated for complex Hessian operators.

{\bf Step 1. Construction of subsolution.} To establish existence, we first construct a subsolution to \eqref{Dirichlet problem}. Fix $\lambda \in (\Lambda_0, \Lambda_1)$. By Lemma \ref{short time existence}, there exists a $\Gamma$-admissible solution $u_{\lambda}$ to \eqref{family}. For any constant $N >\|\psi(\cdot,0)\|_{L^{\infty}} > 0$, define $u_{N} = N u_{\lambda}$. Then, using the homogeneity of $F$ and the assumption $\psi_{s} \geq -\Lambda_0>-\lambda$, we obtain
\begin{equation}\label{sub}
F(D^2 u_{N}) =N-\lambda u_{N}\geq \|\psi(\cdot,0)\|_{L^{\infty}}-\lambda u_{N}>\psi(x,u_{N})\,,
\end{equation}
which shows that $u_{N}$ is a subsolution of \eqref{Dirichlet problem}.

{\bf Step 2. Continuity method.} We now prove existence via the continuity method. For a parameter $t \in [0,1]$, consider the family of equations
\begin{equation}\label{family 2}
F(D^2u_{t})= t\psi(x,u_{t}) + (1-t)F(D^2u_{N})\, \quad \textrm{in} \, \Omega, \qquad u_{t}=0 \quad \textrm{on} \, \partial \Omega.
\end{equation}
Define the set
\[
S=\{t\in [0,1]\,:\, \eqref{family 2} \textrm{ admits a solution } u_{t}\in C^{\infty}(\overline{\Omega})\}\,.
\]
Note that at $t = 0$, the function $u_N$ is a solution to \eqref{family 2}, so $0 \in S$ and $S \neq \emptyset$.  To apply the continuity method, we shall verify that $S$ is open and closed.

{\bf Step 3. Openness.}  Let
\[
L_{t}=\Delta_{F}-t\psi_{s}(x,u_{t}) \qquad \textrm{ where} \quad \Delta_{F}:=\sum_{i,j=1}^{n}F^{ij}(D^2u_{t})\partial_{i}\partial_{j}
\]
denote the linearization of \eqref{family 2}. By concavity, 
\[
\Delta_{F}(u_{N}-u_{t})\geq F(D^2u_{N})-F(D^2u_{t})\,.
\]
Moreover, from \eqref{sum} applied to $u = u_t$, we get $\Delta_{F}u_{t}=F(D^2u_{t})$. This implies 
\[
\Delta_{F}u_{N} \geq F(D^2u_{N})=N-\lambda u_{N}> -\lambda u_{N}.
\]
It follows that
\begin{equation}
\lambda_{1}(\Delta_{F})> \lambda \geq t\lambda>-t\psi_{s}\,,
\end{equation}
where $\lambda_{1}(\Delta_{F})$ is the first eigenvalue of the linear operator $\Delta_{F}$. Therefore, the operator $L_t$ is invertible, which implies that $S$ is open. 

{\bf Step 4. Closeness.} To prove that $S$ is closed, it suffices to establish a uniform $L^\infty$ estimate for the solutions $u_t$. Once this is achieved, the higher-order estimates follow from the arguments in Sections \ref{First order estimate} and \ref{Second order estimates}.  

From \eqref{sub} and \eqref{family 2}, we obtain
\begin{equation}\label{zuo cha}
F(D^2u_{t})-F(D^2u_{N})<t(\psi(x,u_{t})-\psi(x,u_{N})).
\end{equation}
Define $u_{r}=ru_{t}+(1-r)u_{N}$ for $r\in (0,1)$, and let
\[
\overline{F}^{ij}=\int_{0}^{1}F^{ij}(D^2u_{r})dr\,, \quad \overline{\psi}_{s}=\int_{0}^{1}\psi_{s}(x,u_{r})dr\,, \quad \Delta_{\overline{F}}=\sum_{i,j=1}^{n}\overline{F}^{ij}\partial_{i}\partial_{j}\,.
\]
Set $w=u_{N}-u_{t}$. Then \eqref{zuo cha} simplifies to $\Delta_{\overline{F}}w>t\overline{\psi}_{s}w\,,$
which also implies that
\begin{equation}\label{bar F eigen}
\lambda_{1}(\Delta_{\overline{F}})>-t \overline{\psi}_{s}\,.
\end{equation}
Let $v_{1}$ be the first eigenfunction of $\Delta_{\overline{F}}$, normalized so that $\inf_{\Omega}v_{1}=-1$, and define
\[
\theta_{0}=\inf\{\theta\,:\, w+\theta v_{1}\leq 0\, \textrm{ in } \Omega\}.
\]
We claim that $\theta_{0}\leq 0$. Suppose, for contradiction, that $\theta_{0}> 0$. Then from \eqref{bar F eigen}, we derive
\begin{equation}\label{contrary}
\Delta_{\overline{F}}(w+\theta_{0}v_1)>t\overline{\psi}_{s}w-\lambda_{1}(\Delta_{\overline{F}})\theta_{0}v_{1}>t\overline{\psi}_{s}(w+\theta_{0}v_1).
\end{equation}
Choose a  positive constant $A$ such that $A\geq t\overline{\psi}_{s}$. Since $w+\theta_{0}v_1\leq 0$ by the definition of $\theta_0$, it follows that
\[
(\Delta_{\overline{F}}-A)(w+\theta_{0}v_1)\geq 0.
\]
The strong maximum principle then implies that either $w+\theta_{0}v_1\equiv 0$ or $w+\theta_{0}v_1<0$ in $\Omega$. The former case contradicts \eqref{contrary}. For the latter case, by the Hopf lemma for $\Delta_{\overline{F}}$, there exists $\theta_{1}>0$ such that $w+\theta_{0}v_1\leq \theta_{1} v_{1}$ in $\Omega$, which contradicts the definition of $\theta_{0}$. This proves the claim. Therefore, we conclude that $w\leq -\theta_0v_{1}\leq  0$, and hence
\[
0\geq u_{t}\geq u_{N} \quad \textrm{ in } \Omega.
\]
Thus, $S$ is closed. 

Consequently, since $S$ is nonempty, open and closed, it follows that  $S=[0,1]$. This completes the proof.
\end{proof}

\begin{remark}
When $\psi_{s} \geq 0$, Proposition \ref{bif} partially recovers the classical theorem of Caffarelli–Nirenberg–Spruck \cite{CNS85}. Our result thus relaxes their structural assumption.  However, due to the lack of a variational structure for the general concave elliptic Hessian operator $F$, the question of uniqueness remains an open challenge.
\end{remark}

\end{document}